\documentclass{amsart}
\usepackage{amsmath,amsthm,amssymb,amsfonts,amscd}
\usepackage[french]{babel}
\usepackage[utf8]{inputenc}
\usepackage[T1]{fontenc}

\usepackage{tikz}

\usepackage{amssymb,latexsym}
\usepackage{color}

\usepackage{amsthm}
\usepackage{eucal}

\title[Spectre des variétés hyperboliques]{Sur le spectre et la topologie des vari\'et\'es hyperboliques de congruence : les cas complexe et quaternionien}

\author{Nicolas Bergeron et Laurent Clozel} 
\address{Institut de Math\'ematiques de Jussieu \\
Unit\'e Mixte de Recherche 7586 du CNRS \\
Universit\'e Pierre et Marie Curie \\
4, place Jussieu 75252 Paris Cedex 05, France \\}

\email{bergeron@math.jussieu.fr}
\urladdr{http://people.math.jussieu.fr/~bergeron}

\address{Universit\'e Paris Sud \\
Unit\'e Mixte de Recherche 8628 du CNRS \\
Laboratoire de Math\'ematiques \\
B\^at. 425, 91405 Orsay cedex, France\\}

\email{Laurent.Clozel@math.math.u-psud.fr}

\newtheorem{thm}{Th\'eor\`eme}[section]   
\newtheorem{lem}[thm]{Lemme}

\theoremstyle{definition}

\numberwithin{equation}{section}

\newcommand{\RR}{\mathbf R}
\newcommand{\CC}{\mathbf C}
\newcommand{\ZZ}{\mathbf Z}
\newcommand{\QQ}{\mathbf Q}
\newcommand{\Ab}{\mathbf{A}}

\newcommand{\GL}{\mathrm{GL}}
\newcommand{\SL}{\mathrm{SL}}

\newcommand{\SO}{\mathrm{SO}}
\newcommand{\Sp}{\mathrm{Sp}}
\newcommand{\SU}{\mathrm{SU}}
\newcommand{\U}{\mathrm{U}}
\newcommand{\G}{\mathbf{G}}
\newcommand{\GM}{\mathbf{M}}

\def\adots{\mathinner{\mkern2mu\raise1pt\hbox{.}
\mkern3mu\raise4pt\hbox{.}\mkern1mu\raise7pt\hbox{.}}}

\begin{document}

\begin{abstract}  
En nous basant sur les r\'esultats d'Arthur et de Mok, 
nous étendons aux variétés hyperboliques de volume fini complexes et quaternioniennes les résultats de \cite{Inventiones}. 
Dans le cas du spectre sur les fonctions, nous montrons que nos résultats de \og quantification \fg \ des valeurs propres sont optimaux. 
En guise d'application, on démontre enfin une \og propriété de Lefschetz \fg \ pour l'application de restriction en cohomologie d'un quotient arithmétique 
non compact d'une boule vers un quotient d'une boule de dimension plus petite.  Ce résultat généralise un résultat récent de Nair et en donne une version \og optimale \fg.
\end{abstract}
\maketitle

\section{Introduction}
Dans cette note, nous étendons aux variétés hyperboliques de volume fini complexes et quaternioniennes les résultats de \cite{Inventiones}. On considère donc des quotients de la forme $\Gamma\backslash G/K$ où $G$ est (isogène à) $\SU (n,1)$ ou $\Sp(n,1)$, $K$ en est un sous-groupe compact maximal, et $\Gamma\subset G$ est un sous-groupe de congruence (notion précisée plus loin).

En premier lieu, nous décrivons, dans le cas unitaire et en tout degré $k$, le spectre du laplacien dans les $k$-formes. Les résultats sont exactement analogues à ceux de \cite{Inventiones}~: les plus petites valeurs propres sont des entiers explicites ; les autres valeurs propres admettent une borne inférieure que l'on détermine, d'une part (ceci avait été fait dans \cite{Asterisque} dans le cas unitaire) en supposant la \og conjecture de Ramanujan\fg\ pour $\GL(n)$ puis --- argument inconditionnel --- à l'aide de l'approximation connue de celle-ci. Dans le cas quaternionien, nos résultats sont moins précis en ce sens que nous ne décomposons pas le spectre selon les différents $M$-types.\footnote{Les résultats de Pedon \cite{Pedon} montrent que ceci est déjà tr\`{e}s difficile pour le spectre tempéré.} Dans les deux cas nos résultats impliquent toutefois le théorème suivant. 

\begin{thm} \label{T0}
Il existe une nombre r\'eel $\varepsilon = \varepsilon (G) <1$, que l'on peut prendre égal à $0$ si on suppose la conjecture de Ramanujan, tel que pour tout $\Gamma$, si $\lambda>0$ est une valeur propre du laplacien dans les $k$-formes de $\Gamma\backslash G/K$, alors
\begin{enumerate}
\item $\lambda$ est un entier pair, ou
\item $\lambda \geq \alpha_k - \varepsilon$, où $\alpha_k$ est la borne inférieure du spectre (dit tempéré) du laplacien dans les $k$-formes de carré intégrable sur $G/K$.
\end{enumerate}
\end{thm}
Noter que $\alpha_k \geq 1$.

Dans le cas du spectre sur les fonctions on obtient en outre une énumération optimale des entiers qui peuvent intervenir dans le premier cas. Les démonstrations sont ici en tout point similaires à celles de \cite{Inventiones}, dont nous n'avons pas reproduit enti\`{e}rement les arguments relatifs à la formule des traces : on les retrace brièvement, pour les groupes unitaires, dans le \S~\ref{S1} et pour les groupes $\Sp(n,1)$ dans le \S~\ref{S3}. Rappelons le point clé : à l'aide de la formule de Masushima (généralisée, cf. \cite[Ch.1]{Asterisque}) le spectre du laplacien se ramène au calcul du caractère infinitésimal des représentations de $G$ apparaissant dans $L^2(\Gamma\backslash G)$, et ce calcul, à son tour, peut être fait pour la forme quasi-déployée  $G^*$ de $G$. Or, quand $G$ est de rang 1, les propriétés d'intégralité du caractère infinitésimal dues à l'existence d'un grand sous-groupe compact maximal restreignent considérablement les représentations de $G^*$ qui nous concernent. Par ailleurs, le spectre discret de $G^*$ a été déterminé par Mok \cite{Mok} (cas unitaire) et par Arthur \cite{ArthurBook} (cas des groupes orthogonaux quaternioniens) : leurs résultats nous permettent de conclure.\footnote{N.B. :  dans tous les calculs relatifs aux laplaciens, il s'agit du laplacien \emph{positif} : celui (normalisé de la façon indiquée) dont les valeurs propres sont positives.}

Dans le \S~\ref{S2}, nous décrivons quelques conséquences topologiques du théorème \ref{T0}. Ces conséquences ont déjà été annoncées dans le cas unitaire lorsque le groupe est anisotrope; on tirerait des conséquences similaires dans le cas quaternionien. Nous préférons nous restreindre au cas unitaire mais traitons plus particulièrement 
le cas isotrope; cela nous permet en effet de généraliser un résultat récent d'Arvind Nair \cite{Nair}. Le théorème obtenu est une version optimale de la \og propriété de Lefschetz \fg \ pour l'application de restriction entre variétés de Shimura unitaires; cf. Théorème~\ref{lef}.

Comme on l'a dit, nos démonstrations sont inconditionnelles ; on a soigneusement isolé les estimées, meilleures, reposant sur la conjecture de Ramanujan. En revanche, notre point de départ --- la description du spectre des groupes quasi-déployés, due à Arthur et Mok --- repose, en l'état actuel de la théorie, sur la stabilisation supposée de la formule des traces (de Selberg) tordue. Voir \cite[Chapitre 4]{ArthurBook} pour une discussion plus précise. Noter que cette stabilisation est maintenant annoncée par Waldspurger dans son exposé au congrès international de Séoul; la stabilisation attendue est conséquence de travaux de Waldspurger et Moeglin-Waldspurger. 

\section{Résultats spectraux : groupes unitaires} \label{S1}

Dans ce paragraphe $G=\SU(n,1)$, $K$ est son sous--groupe compact maximal $\mathrm{S}(\mathrm{U}(n)\times \mathrm{U}(1))$, et $X=G/K$.  Noter que $X$ est aussi égal à $\mathrm{U}(n,1)/\mathrm{U}(n)\times \mathrm{U}(1)$. Une variété hyperbolique de congruence est un quotient de $X$ obtenu de la façon suivante : soit $F$ une extension totalement réelle de $\QQ$, $E$ une extension quadratique $CM$ de $F$, et $G$ un $F$-groupe dont l'extension des scalaires $G_E$ est le groupe des unités d'une algèbre simple centrale qui peut \^{e}tre l'alg\`{e}bre des matrices. On suppose de plus que $G(F\otimes_\QQ \RR)$ est isomorphe à $\mathrm{U}(n,1) \times \mathrm{U}(n+1)^{d-1}$ où $d=[F:\QQ]$; on désigne par $v_0$ la place de $F$ correspondant au facteur non compact. Si $\Gamma_0 \subset G(F)$ est un sous-groupe de congruence, et $\Gamma$ sa projection sur $\mathrm{U}(n,1)$, $\Gamma\backslash X$ est un espace hyperbolique de congruence. (La différence entre les groupes unitaire et spécial unitaire n'est pas ici pertinente.)

Soit $B$ la forme de Killing de $\mathfrak{g}_0=\mathrm{Lie}_\RR(G)$. Un calcul standard montre que $B(X,Y)=2 (n+1) \mathrm{Tr}(XY)$ ($X,Y \in \mathfrak{g}_0$). 
On préfère --- comme dans \cite[Ch.~4]{Asterisque} --- renormaliser $B$ et considérer la forme invariante 
$$\langle X, Y\rangle = \frac12 \mathrm{Tr} (XY).$$
Noter qu'en identifiant de manière naturelle l'espace tangent \`a $X$ en $eK$ \`a l'espace vectoriel $\CC^n$, la forme invariante $\langle , \rangle$ induit sur $\CC^n$ sa forme 
hermitienne usuelle. La forme $\langle , \rangle$ définit canoniquement la structure hermitienne (donc riemannienne) sur $X$ ainsi que le laplacien (riemannien) $\Delta_k$ sur les $k$-formes.
Pour cette normalisation les courbures sectionnelles de $X$ sont comprises entre $-4$ et $-1$. Rappelons enfin que $G$ contient un parabolique $P=MAN$ avec $A\cong \RR_+^\times$ et $M\cong \mathrm{U}(1)\times \mathrm{U}(n-1)$ \cite[Ch.~4]{Asterisque}. Enfin, on pose $N=n+1$ (malgré le conflit de notation) et $\eta_N=\frac{1}{2} - \frac{1}{N^2+1}$.

Fixons $k\le n$, ainsi qu'un type de Hodge $(p,q)$ avec $p+q=k$. Les représentations unitaires irréductibles de $G$ qui contribuent aux formes sur $\Gamma\backslash X$ de type $(p,q)$ sont des séries discrètes (qui n'interviennent pas ici, la valeur propre correspondante de $\Delta_k$ étant nulle) ou sont contenues dans l'induite d'une représentation de $M$, donnée par des paramètres~$a,b$~: \cite[\S~4.5]{Asterisque}.

\begin{thm} \label{T1}
Les valeurs propres $\lambda$ de $\Delta_k$ dans les formes de type $(p,q)$ appartiennent à l'ensemble suivant :
\begin{itemize}
\item[(i)] $\lambda=(n-a-b)^2 - (n-a-b-2k)^2$ où $a\le p$, $b\le q$, $p-q-(a-b)\in \{-1,0,1\}$ et $0\le k\le [\frac{n-a-b}{2}]$
\medskip
\item[(ii)] $\lambda\ge (n-a-b)^2-4\eta_N^2=(n-a-b)^2-(\frac{N^2-1}{N^2+1})^2$, $a,b$ comme en (i).

Si on suppose la conjecture de Ramanujan, (ii) est remplacée par :
\medskip
\item[(iii)] $\lambda\ge (n-a-b)^2$, $a,b$ comme en (i).
\end{itemize}
\end{thm}
\begin{proof}[Esquissons la démonstration] Supposons d'abord $G$ anisotrope (comme $F$-groupe). Notons $G^*$ la forme intérieure de $G$ sur $F$, et parfois par abus de notation sur $\RR$. Ainsi $G^*(\RR)=\mathrm{U}(r,s)$ où $|r-s|\le 1$, $r+s=n+1$. Mok \cite{Mok}, étendant les résultats d'Arthur au cas unitaire, a paramétré les représentations de $G^*(F_\infty)$ apparaissant dans le spectre discret (pour $G^*(F)$) par des paramètres d'Arthur relatifs à $\mathrm{U}(r,s)$. On  renvoie à \cite[Chap.~6]{Asterisque} pour une description détaillée. Un tel paramètre, $\psi$, définit un caractère infinitésimal pour $\mathrm{U}(r,s)$ et aussi pour $\mathrm{U}(n,1)$ \cite[Lemme~6.1]{Asterisque}. Il définit aussi un paquet d'Arthur $\Pi(\psi)$ de représentations irréductibles de $G^*(\RR)=\mathrm{U}(r,s)$. On a alors (cf. \cite[Lemme~3.4]{Inventiones})~:

\begin{lem}
Si $\pi \in \Pi(\psi)$, le caractère infinitésimal $\nu$ de $\pi$ est celui associé~à~$\psi$.
\end{lem}

Ceci résulte de l'identité de caractères  entre (paquet de) représentations de $G^*(\RR)$ et la représentation correspondante de $\GL(N,\CC)$, cf. Mok \cite[Théorème 2.5.1]{Mok}. Soit $\mathcal{Z}_{\mathbf{R}}$ le centre de l'alg\`{e}bre enveloppante de $\mathfrak{g}_0\otimes \mathbf{C}$ et $\mathcal{Z}_{\mathbf{C}}$ l'analogue pour $\GL(N,\mathbf{C})$. Il existe un homomorphisme surjectif $N: \mathcal{Z}_{\mathbf{C}} \rightarrow \mathcal{Z}_{\mathbf{R}}$ \cite{ClozelThese}. Un param\`{e}tre $\nu$ du dual d'une alg\`{e}bre de Cartan pour $G^*(\mathbf{R})$, d\'{e}finissant, {\it via} l'homomorphisme d'Harish-Chandra, un caract\`{e}re de $\mathcal{Z}_{\mathbf{R}}$, donne par composition un caract\`{e}re de $\mathcal{Z}_{\mathbf{C}}$, qu'on lui identifie. Notons $h_{\mathbf{R}}$, $h_{\mathbf{C}}$ les homomorphismes d'Harish-Chandra.

Il r\'{e}sulte des identités de caract\`{e}res dans le cas tempéré \cite{ClozelThese} et des arguments donnés dans \cite[Chapitre 1]{ArthurClozel} et \cite[Appendice]{ClozelDelorme} que, pour $\varphi$ sur $G(\mathbf{C})$ et $f$ sur $G( \mathbf{R})$ associées par les identités de changement de base stable, $z\varphi$ et $(Nz)f$ sont associées ($z \in \mathcal{Z}_{\mathbf{C}})$.

Considérons alors l'identité de caract\`{e}res de Mok \cite[Théor\`{e}mes 2.5.1, 3.21]{Mok}~:
\begin{equation} \label{E1}
\mathrm{trace} (\Pi (\varphi) I_{\sigma}) = \sum_{\pi} \varepsilon(\pi) m(\pi) \mathrm{trace} \ \pi(f).
\end{equation}
Ici $\Pi$ est l'unique représentation de $\GL(N,\mathbf{C})$ associée \`{a} $\psi$, $\pi$ parcourt $\Pi(\psi)$,  $m(\pi)>0$ est une multiplicité, et $\varepsilon(\pi)$ est un signe. Si l'on remplace $\varphi$ par $z\varphi$ dans \eqref{E1}, on en déduit, $\nu$ étant le param\`{e}tre infinitésimal de $\Pi$~:
\begin{equation} \label{E2}
\sum_{\pi} \varepsilon(\pi) \mathrm{trace} \ \pi(zf) = \sum_{\pi} \varepsilon(\pi) h_{\mathbf{R}}(z) (\nu) \mathrm{trace} \ \pi (f)
\end{equation}
pour toute $f$ associée \`{a} $\varphi$. Pour les groupes unitaires, la norme entre classes de conjugaison tordue dans $G(\mathbf{C})$ et classes de conjugaison stable dans $G(\mathbf{R}$) est surjective. L'identité \eqref{E2} détermine donc la somme de caract\`{e}res de droite au voisinage de tout élément régulier de $G(\mathbf{R})$. Les caract\`{e}res étant linéairement indépendants et $\varphi \leadsto f$ localement surjective, il en résulte que le caract\`{e}re infinitésimal de chaque $\pi$ est égal \`{a} $\nu$. 

Nous pouvons maintenant imiter mot pour mot les arguments de \cite[\S~4.5]{Inventiones}. A l'issue de la 
stabilisation (\S~4) puis de la déstabilisation (\S~5) de la formule des traces pour $G(F)$, on voit que le caractère infinitésimal $\nu$ d'une représentation de $G(F_{v_0})=\mathrm{U}(n,1)$ est la somme de caractère $\nu_i$ de représentations associées à des paramètres $\psi_i$, de rang inférieur, de groupes unitaires quasi-déployés $G_i^*$ sur $F$ de rangs $N_i$ $(\sum N_i=N)$. Mais la somme directe des $\psi_i$ est un paramètre $\psi$ pour $G^*$. Ainsi :

\begin{lem}
Si $\pi$ est une représentation de $G(F_{v_0})=\mathrm{U}(n,1)$ apparaissant dans $L^2(\Gamma\backslash G(F_{v_0}))$, le caractère infinitésimal $\nu$ de $\pi$ provient d'un paramètre $\psi$ pour $\mathrm{U}(r,s)$.
\end{lem}

Nous sommes alors ramenés à la démonstration donnée dans \cite[\S~6.2]{Asterisque}. Rappelons que le paramètre archimédien $\psi_\CC$ définit par restriction à $\CC^\times \times \SL(2,\CC)$
\begin{equation} \label{E3}
\psi_\CC = \sum_j \chi_j \otimes r_j
\end{equation} 
où $\chi_j$ est un caractère de $\CC^\times$, $r_j$ une représentation de degré $n_j$ de $\SL(2,\CC)$, et $n+1=\sum n_j$. Les caractères $\chi_j$ vérifient l'estimée connue de Ramanujan pour les représentations cuspidales, {\it i.e.}
\begin{equation} \label{E4}
\chi_j(z) = z^{p_j}\overline{z}^{q_i},\ |\mathrm{Re}\frac{(p_j+q_j)}{2}| \leq \frac{1}{2} -\frac{1}{N^2+1}\ (N=n+1). 
\end{equation} 
La conjecture de Ramanujan est $\mathrm{Re}(p_j+q_j)=0$.

Le caractère infinitésimal déduit de \eqref{E3}, un élément de $\CC^{n+1}/\mathfrak{S}_{n+1}$, est
\begin{equation*}
P=(p_j+i), \  j=1,\ldots,r, \ i=\frac{1-n_j}{2}, \frac{3-n_j}{2}, \ldots , \frac{n_j-1}{2},\ i\equiv \frac{n_j-1}{2}\ [1].
\end{equation*}
Or $P$ a $(n-1)$ coordonnées appartenant à $\frac{1}{2}\ZZ$. Si toutes ses coordonnées sont des demi-entiers, la démonstration du théorème est donnée dans \cite[\S~6.2]{Asterisque} ; on est alors dans le cas (i). Dans le cas inverse, 
il pourrait intervenir dans la décomposition \eqref{E3} un caractère $\chi=\chi_i$ avec $p=p_i\notin \frac{1}{2}\ZZ$, $n_i$ étant égal à $2$ ; et deux caractères $\chi,\chi'$ avec multiplicité 1. Le premier cas est éliminé (inconditionnellement) dans \cite[p.~70]{Asterisque}. Dans le second, $\chi=z^p(\overline{z})^q$ vérifie la majoration \eqref{E4}, donc est de la forme $(z/\overline{z})^m(z\overline{z})^s$ avec $m\in \frac{1}{2}\ZZ$, $s$ réel soumis à \eqref{E4}.  L'argument du corollaire~6.2.7 de \cite{Asterisque} donne alors la minoration (ii).\footnote{Tel quel le  corollaire 6.2.7 de \cite{Asterisque} est incorrect : il ne donne que les valeurs propres \og discrètes\fg\ et omet le spectre \og continu\fg\ donné par (iii) du théorème~\ref{T1}, qui apparaît évidemment même sous la conjecture de Ramanujan.}

Il nous reste à considérer le cas où $G$ est (globalement) isotrope. Vu la nature de $G(F_\infty)$,  ceci implique que $F=\QQ$ et que $G$ est un groupe unitaire de rang (rationnel) 1 sur $\QQ$, déployé sur un corps quadratique imaginaire $E$. Dans ce cas, $L^2(\Gamma\backslash G(\RR))$ contient un spectre continu, et le théorème~\ref{T1} reste vrai pour celui-ci. Considérons d'abord le spectre discret. Le point de départ des arguments de stabilisation/déstabilisation de \cite[\S~4--5]{Inventiones} est la formule d'Arthur (cf. \cite[\S~5.7]{Inventiones})
\begin{equation*}
\begin{split}
I_{{\rm disc},t}^G(f) & := \mathrm{trace}(f\mid L_{{\rm disc},t}^2(G(F)\backslash G(\Ab_F)) \\
& \quad \quad +\sum_M \frac{|W_0^M|}{|W_0^G|} \sum_{s\in W(M)_{reg}}|\det(s-1)|^{-1} \mathrm{trace}(M_p(s,0)I_{p,t}(0,f))\\
&=\sum_H \iota (G,H) S_{disc,t}^H (f^H)\,.\hskip5cm
\end{split}
\end{equation*}
Rappelons qu'ici $t>0$ paramètre la norme de la partie imaginaire du caractère infinitésimal. On renvoie à Arthur ainsi qu'à \cite[\S~4, \S~6]{Inventiones} pour les détails. L'argument précédent a consisté à vérifier que la somme endoscopique $\sum_H$ est concentrée en un
caractère infinitésimal donné par un paramètre d'Arthur, donc impliquant les estimées (i) ainsi que (ii) ou (iii) selon le cas. Pour obtenir le résultat pour le spectre discret, il nous suffit donc de comprendre les caractères infinitésimaux apparaissant dans le terme complémentaire (somme sur les sous-groupes de Levi propres) de l'expression de $I_{{\rm disc},t}^G(f)$. Il n'apparaît ici qu'un seul sous-groupe, $M\cong \GL(1,E) \times \mathrm{U}(n-1)$, le groupe $\mathrm{U}(n-1)$ étant anisotrope. La représentation induite est somme d'induites unitaires $\mathrm{ind}_{P(\Ab)}^{G(\Ab)}\chi\otimes \tau$, $P$ étant un parabolique de sous-groupe de Levi $M$, $\chi$ un caractère de $E^\times\backslash\Ab_E^\times$ et $\tau$ une représentation (cuspidale) de $\mathrm{U}(n-1,\Ab)$. Il n'y a qu'un élément régulier dans $W(M)$, envoyant $(\chi,\tau)$ vers $(\chi^{-c},\tau)$ où $c : E^\times \rightarrow E^\times$ est la conjugaison complexe. Le terme complémentaire concerne donc les données $(\chi,\tau)$ telles que $\chi^c =\chi^{-1}$, {\it i.e.} $\chi$ est une donnée unitaire (au sens de Mok) pour $\GL(1,E)$. (Le paramètre complet d'Arthur--Mok est alors $\chi \boxplus \chi\boxplus \tau_E$ où $\tau_E$ est la représentation, associée à $\tau$, de $\GL(n-1,\Ab_E)$). Le terme complémentaire est donc associé à une donnée d'Arthur--Mok, et les arguments précédents permettent de conclure. Enfin, une représentation du spectre continu est de même, d'après Langlands, induite à partir d'une donnée $(\chi,\tau)$ où $\chi$ est maintenant un caractère unitaire de $E^\times\backslash \Ab_E^\times$ {\it au sens usuel} : $|\chi(z)|=1$, $z\in \Ab_E^\times$. La donnée associée est alors $\chi\boxplus\chi^{-c} \boxplus \tau_E$ et a les mêmes propriétés. 
\end{proof}

\medskip

On peut expliciter un peu plus le théorème \ref{T1}. 
L'espace sym\'etrique $X$ est un domaine born\'e de $\CC^{N}$. Soit $k (z_1 , z_2)$ le noyau de Bergmann de $X$. Il lui correspond la forme de K\"ahler 
$\omega (z) = \partial \overline{\partial} \log k(z,z)$. Chaque quotient $\Gamma \backslash X$ est une vari\'et\'e complexe hermitienne sur lequel la forme $(1/2i \pi) \omega$
induit une forme de type $(1,1)$. En plus des op\'erateurs $\partial$, $\overline{\partial}$ et de leurs adjoints, on dispose donc de l'op\'erateur {\it de Lefschetz} $L$, cup-produit avec $\omega$, et de son adjoint $\Lambda$ sur l'espace $\Omega^*_{(2)} (\Gamma)$ des formes diff\'erentielles de carr\'e int\'egrable sur $\Gamma \backslash X$. 

Fixons maintenant un entier $r \leq n$, ainsi qu'un type de Hodge $(p,q)$ avec $p+q = r$. Le sous-espace $\Omega^{p,q}_{(2)} (\Gamma )$ des formes diff\'erentielles de type $(p,q)$
se d\'ecompose en 
\begin{equation} \label{dec1}
\Omega^{p,q}_{(2)} (\Gamma) = \oplus_{k=0}^{\min (p,q)} L^k (P^{p-k,q-k}_{(2)} (\Gamma)), 
\end{equation}
o\`u $P^{p-k,q-k}_{(2)} (\Gamma) = \Omega^{p-k,q-k}_{(2)} (\Gamma) \cap \mathrm{Ker} \Lambda$ d\'esigne le sous-espace des formes {\it primitives} dans $\Omega^{p-k,q-k}_{(2)} (\Gamma)$. 

Il d\'ecoule enfin des identit\'es de Hodge que le sous-espace $P^{p,q}_{(2)} (\Gamma)$ se d\'ecompose lui-m\^eme en une somme directe 
\begin{equation} \label{dec2}
P^{p,q}_{(2)} (\Gamma) = K^{p,q}_{(2)} (\Gamma ) \oplus \partial (K^{p-1,q}_{(2)} (\Gamma )) \oplus \overline{\partial} (K^{p,q-1}_{(2)} (\Gamma )) \oplus \partial \overline{\partial} (K^{p-1 , 
q-1}_{(2)} (\Gamma)) , 
\end{equation}
o\`u $K^{p,q}_{(2)} (\Gamma )$ d\'esigne le sous-espace des formes primitives qui sont \`a la fois $\partial$ et $\overline{\partial}$ coferm\'ees, autrement dit
\[K^{p,q}_{(2)} (\Gamma ) = P^{p,q}_{(2)} (\Gamma) \cap \mathrm{Ker} (\partial^*) \cap \mathrm{Ker} (\overline{\partial}^*) .\]

Puisque les d\'ecompositions \eqref{dec1} et \eqref{dec2} sont invariantes par le laplacien, pour d\'ecrire compl\`etement le spectre de ce dernier il suffit de d\'ecrire le spectre
sur les formes primitives de degr\'e $r<n$ qui sont \`a la fois $\partial$ et $\overline{\partial}$ coferm\'ees. C'est l'objet du th\'eor\`eme suivant qui se déduit du théorème \ref{T1}.

\begin{thm} \label{TU}
Soient $r$ un entier $< n$ et $(p,q)$ un type de Hodge avec $p+q = r$. Les valeurs propres $\lambda$ du laplacien $\Delta_r$ dans les formes {\rm primitives} de type $(p,q)$
qui sont \`a la fois {\rm $\partial$ et $\overline{\partial}$ coferm\'ees} appartiennent \`a l'ensemble 
\[ \bigcup_{k=0}^{[\frac{n-r}{2}]} \left\{ (n-r)^2 - (n-r-2k)^2 \right\} \bigcup \left[(n-r)^2 - \left(\frac{N^2-1}{N^2+1}\right)^2 , +\infty \right[ .\] 
\end{thm}
Si on suppose la conjecture de Ramanujan, on peut remplacer l'intervalle par $[(n-r)^2 , + \infty [$. 

\begin{proof}[D\'emonstration du th\'eor\`eme \ref{TU}]
Notons $\tau_r$ la repr\'esentation adjointe de $K$ dans le produit ext\'erieur $\wedge^r$ du complexifi\'e de l'espace tangent \`a $X$. 
Rappelons (voir par exemple \cite[Chapitre 4]{Asterisque}) que la repr\'esentation $\tau_r$ se d\'ecompose en $\oplus_{p+q=r} \tau_{p,q}$ et que chaque $\tau_{p,q}$ se d\'ecompose en irr\'eductibles 
$\oplus_k \tau_{p-k, q-k}'$ de sorte que la formule de Matsushima ram\`ene le calcul du spectre du laplacien dans $P^{p,q}_{(2)} (\Gamma)$ au calcul 
du caract\`ere infinit\'esimal des repr\'esentations $\pi$ de $G$ apparaissant dans $L^2 (\Gamma \backslash G)$ et telles que 
$\mathrm{Hom}_K (\tau_{p,q}' , \pi ) \neq \{ 0 \}$. 

Maintenant les repr\'esentations unitaires irr\'eductibles de $G$ qui contribuent aux formes dans $P^{p,q}_{(2)} (\Gamma)$ sont soit des s\'eries discr\`etes (qui n'interviennent pas ici, la valeur propre correspondante du laplacien $\Delta_r$ \'etant nulle) ou induites d'une repr\'esentation de $M\cong 
\U (1) \times \U (n-1)$. Il d\'ecoule de plus de la réciprocité de Frobenius que la repr\'esentation de $M$ est un $M$-type de $\tau_{p,q}'$. Or on a~:
$$(\tau_{p,q}')_{| M} = \sigma_{p,q} \oplus \sigma_{p-1,q} \oplus \sigma_{p,q-1} \oplus \sigma_{p-1 , q-1},$$
o\`u chaque $\sigma_{a,b}$ est une repr\'esentation irr\'eductible de $M$. Il d\'ecoule finalement de \cite[Th\'eor\`eme 9.4.3]{Asterisque} qu'une repr\'esentation
induite de $\sigma_{a,b}$ contribue \`a  $K^{p,q}_{(2)} (\Gamma)$ pr\'ecis\'ement quand $(a,b)=(p,q)$. Le th\'eor\`eme \ref{TU} se déduit alors de la d\'emonstration du théorème \ref{T1} telle qu'esquissée ci-dessus.
\end{proof}

Contrairement au cas hyperbolique r\'eel, nous ne connaissons pas de preuve que toutes les valeurs propres discr\`etes peuvent effectivement intervenir 
dans un quotient. Dans \cite[Th\'eor\`eme 6.5.1]{Asterisque} on donne une telle preuve lorsque $p=q$. Le cas g\'en\'eral devrait toutefois r\'esulter des r\'esultats d'Arthur et Mok mais 
nous ne l'avons pas vérifié. Modulo la conjecture de Ramanujan, le théorème \ref{TU} est en tout cas optimal pour le spectre du laplacien sur les fonctions, c'est-à-dire qu'il existe une variété hyperbolique complexe de congruence (compacte) 
$\Gamma \backslash X$ dont le spectre du laplacien dans les fonctions contient les valeurs propres 
\[n^2 - (n-2k)^2 \quad (k=0 , \ldots , \frac{n}{2}) .\] 


\section{Applications topologiques} \label{S2}

Il d\'ecoule en particulier du th\'eor\`eme \ref{TU} que la conjecture $A^-$ de \cite{Asterisque} --- ou encore la conjecture 2.3 de \cite{IMRN} --- est v\'erifi\'ee~:
pour tout entier $r \leq n$ la premi\`ere valeur propre non nulle $\lambda_1^r = \lambda_1^r (\Gamma)$ v\'erifie
\begin{equation} \label{trou}
\begin{split}
\lambda_1^r & \geq 4(n-r-1), \mbox{ si } r \neq n\pm 1, n, \mbox{ et } \\
& \geq 1 - \left(1-\frac{2}{(n+1)^2+1}\right)^2 = \left(\frac{2(n+1)}{(n+1)^2 +1} \right)^2, \mbox{ si } r = n\pm 1 \mbox{ ou } n.
\end{split}
\end{equation}

Ce r\'esultat a un certain nombre de cons\'equences sur la topologie des vari\'et\'es $\Gamma \backslash X$; cons\'equences que nous avons rassembl\'ees sous le nom de 
{\it propri\'et\'es de Lefschetz automorphes} dans \cite[Conjecture 1.1]{IMRN}. On renvoie \`a cet article pour les \'enonc\'es et leurs d\'emonstrations lorsque le quotient $\Gamma 
\backslash X$ est compact. On se contente ici de d\'etailler ce qui reste vrai lorsque le quotient $\Gamma \backslash X$ n'est plus suppos\'e compact. On est alors r\'eduit
\`a une situation concr\`ete, \'etudi\'ee en particulier par Nair \cite{Nair}~:

Soient $E\subset \CC$ un corps quadratique imaginaire, $V$ un espace vectoriel de dimension $n+1\geq 3$ sur $E$ et $h :V \times V \to E$ une forme hermitienne,
relativement \`a la conjugaison de $E/\QQ$, telle que $h$ soit de signature $(n,1)$ sur $V_{\RR} := V \otimes_{\QQ} \RR \cong \CC^{n+1}$. Un sous-groupe de congruence
$\Gamma \subset \U (h)$ op\`ere alors proprement discontinument sur l'espace sym\'etrique $X$. La vari\'et\'e quasi-projective $S=S(\Gamma)$ est de volume fini mais non compacte; elle contient des sous-vari\'et\'es de la m\^eme nature~: 
\`a tout sous-espace $W \subset V$ d\'efini sur $E$, de dimension $m+1$ et auquel la forme 
hermitienne se restreint en une forme non d\'eg\'en\'er\'ee et ind\'efinie, on associe en effet un sous-groupe $\U (h_{| W}) \subset \U (h)$. Posant 
$\Gamma_W = \Gamma \cap \U  (h_{| W})$, on obtient une vari\'et\'e quasi-projective $S_W = \Gamma_W \backslash X_W$ o\`u $X_W$ est le sous-espace
sym\'etrique associ\'e au groupe $\U  (h_{| W})$. Il existe en outre un morphisme de vari\'et\'es $S_W \to S$ qui est fini d'image une sous-vari\'et\'e ferm\'ee de codimension $n-m$. 

Oda \cite{Oda} a d\'emontr\'e qu'il existe un ensemble fini de sous-espaces $W_1 , \ldots , W_s$ comme ci-dessus, de dimension $m+1=2$ et tels que 
l'application de restriction 
$$H^1 (S, \QQ) \to \bigoplus_j H^1 (S_{W_j} , \QQ)$$
soit injective. La g\'en\'eralisation de ce r\'esultat \`a des classes de degr\'e $>1$ et \`a des sous-espaces de dimension $>1$ a \'et\'e consid\'er\'e par plusieurs
auteurs, voir \cite{HarrisLi,Venkataramana,IMRN,Nair}. Le th\'eor\`eme est d\'emontré par Nair \cite{Nair} lorsque $i \leq m-2$.

\begin{thm} \label{lef}
Pour tout $m<n$, il existe des sous-espaces $W_1, \ldots , W_s$ de $V$ de dimension $m+1$ tels que l'application de restriction
\begin{equation} \label{res}
H^i (S, \QQ) \to \bigoplus_j H^i (S_{W_j} , \QQ)
\end{equation}
soit injective pour $i\leq m$. 
\end{thm}

On peut reformuler le th\'eor\`eme \ref{lef} \`a l'aide des correspondances de Hecke. \`A tout \'el\'ement rationnel $g$ du groupe $\U (h)$
il correspond en effet une correspondance finie $\Gamma \cap g^{-1} \Gamma g \backslash X \rightrightarrows S$ o\`u la premi\`ere application est la projection
de rev\^etement et la seconde est induite par la translation par $g$. Notons $C_g^* : H^i (S, \QQ) \to H^i (S , \QQ)$ l'endomorphisme induit. \'Etant donn\'e 
un sous-espace $W \subset V$ comme dans le th\'eor\`eme, on montre en fait que 
\begin{quote}
{\it pour toute classe $\alpha \in H^i (S, \QQ)$ de degr\'e $i \leq m$, il existe 
$g$ tel que l'image de $C_g^* (\alpha)$ dans $H^i (S_W , \QQ)$ par l'application de restriction $H^i (S, \QQ) \to H^i (S_{W} , \QQ)$ soit non nulle.} 
\end{quote}
\begin{proof} 
Il suffit de d\'emontrer l'assertion pour des classes de cohomologie complexes. Maintenant si l'on remplace les groupes de cohomologie $H^i (S)$ et $H^i (S_W)$
par les groupes de cohomologie $L^2$ 
$$H^i_{2} (S) = H^i ( \mathfrak{g} , K ; L^2_{\rm disc} (\Gamma \backslash \U (n,1))) \mbox{ et } H^i_{2} (S_W ) = H^i ( \mathfrak{g} , K ; L^2_{\rm disc} (\Gamma_W \backslash \U (m,1)))$$
le r\'esultat est cons\'equence de \cite[Chapitre 9]{MSMF}, se r\'ef\'erer en particulier \`a la d\'emonstration du th\'eor\`eme 9.2, et du trou
spectral \eqref{trou}. 

On conclut en remarquant qu'un th\'eor\`eme de Zucker \cite[Theorem 6.9]{Zucker} implique que l'application naturelle
$H_2^i (S ) \to H^i (S)$ est un isomorphisme si $i \leq n-1$ (en particulier si $i\leq m$) et qu'elle est injective si $i=n$. On applique ce dernier cas \`a
$S_W$ plut\^ot qu'\`a $S$ lorsque $i=m$.  
\end{proof}
Cela implique le th\'eor\`eme \ref{lef}. De la m\^eme mani\`ere on peut d\'eduire de la d\'emonstration de \cite[Th\'eor\`eme 9.4]{MSMF} et du trou spectral 
\eqref{trou} le th\'eor\`eme suivant relatif aux cup-produits dans la cohomologie de $S$.

\begin{thm}
Soient $\alpha$ et $\beta$ deux classes de cohomologie dans $H^{\bullet} (S , \QQ)$ de degr\'es respectifs $k$ et $\ell$ avec $k+\ell \leq n$. Il existe 
alors un \'el\'ement rationnel $g$ du groupe $\U (h)$ tel que 
$$C_g^* (\alpha ) \wedge \beta \neq 0 \mbox{ dans } H^{k +\ell} (S , \QQ).$$
\end{thm}

\section{Résultats spectraux : le cas quaternionien} \label{S3}

\subsection{} Dans ce chapitre on considère le cas où le groupe réel $G$ est isomorphe à $\Sp(n,1)$ \cite[p.~354 de la 1ère édition]{Helgason}. Son sous-groupe compact maximal $K$ est isomorphe à $\Sp(n)\times \Sp(1)$, et l'espace symétrique $X$ est l'espace hyperbolique quaternionien, de dimension réelle $4n$. Nous renvoyons à Pedon \cite{Pedon} 
pour une description plus précise, ainsi que pour la description du spectre {\it tempéré} des formes différentielles sur l'espace quotient; retenons simplement  que Pedon calcule explicitement la borne inférieure $\alpha_k$ du spectre continu du laplacien $\Delta_k$ dans l'espace des $k$-formes $L^2$ sur $X$. Dans tous les cas $\alpha_k$ est un entier strictement positif; il est égal à $(2n+1)^2$
si $k=0$ et égal à $1$ si $k=2n$ (degré médian).

Nous considèrons le quotient de $X$ par un sous-groupe de de congruence $\Gamma$, dans notre acception habituelle. Ici, $\Gamma$ est défini de la façon suivante. Soit $F$ un corps totalement réel, et $\G$ un groupe semi-simple sur $F$ de type (absolu) $C_{n+1}$. On suppose $\G$ défini par une algèbre de quaternions $B$ sur $F$, ramifiée aux places réelles, ainsi que par une forme quaternionienne-hermitienne sur un espace de dimension $(n+1)$ sur $B$, relative à l'involution principale \cite[p. 56]{TitsBoulder}.  On suppose l'indice de Witt de la forme égal à $1$ en une place réelle $v_0$ et à $0$ ailleurs. Ainsi $\G (F\otimes \RR) \simeq \mathrm{Sp}(n,1)\times \mathrm{Sp}(n+1)^{d-1}$ où $d=[F:\QQ]$. On choisit  un groupe de congruence, au sens usuel, $\Gamma_1 \subset G(F)$. Alors $\Gamma_1\backslash \G (F\otimes \RR) =\Gamma \backslash \mathrm{Sp}(n,1)$ où $\Gamma=\Gamma_1 \cap \prod_{v\not=v_0}G(F_v)$. Noter que de tels groupes existent d'après \cite{BorelHarder}. On supposera que $\Gamma$ opère librement, de sorte que $Y=\Gamma\backslash X$ est une variété. (Ceci est inutile si on accepte de travailler dans la catégorie des \og orbifolds\fg). On écrira simplement $G$ pour $\G (F_{v_0})$.

On munit $Y$ de la métrique riemanienne invariante dont les courbures sectionnelles sont comprises entre $-4$ et $-1$ (cf. \cite[p. 234]{Pedon}). Ceci définit le laplacien $\Delta_k$ dans les $k$-formes, $0\le k\le 4n$.

Rappelons les données combinatoires. Tout d'abord, $\G$ est une forme (nécessairement intérieure) du groupe déployé $\mathrm{Sp}(n+1)$. Son groupe dual est $\SO (2n+3,\CC )\times \mathrm{Gal}(\overline{F}/F)$, et de même ${}^LG=\SO(2n+3,\CC) \times \mathrm{Gal}(\CC/\RR)$. En particulier l'entier $N$ tel que $\widehat{G}$ soit naturellement un sous-groupe de $\GL(N,\CC)$ est égal à $2n+3$. Soit $\mathfrak{g}_0=\mathrm{Lie} (G)$ et $\mathfrak{g} =\mathfrak{g}_0\otimes \CC$. On peut identifier naturellement une sous-algèbre de Cartan $\mathfrak{h}$ de $\mathfrak{g}$ à $\CC^{n+1}$, sur laquelle le groupe de Weyl absolu $\Omega\cong \mathfrak{S}_{n+1}\ltimes \{\pm 1\}^{n+1}$ opère de la façon naturelle.

On suppose la forme orthogonale sur $\CC^{2n+3}$, définissant $\widehat{G}$, donnée par la matrice
$$
\widehat{J}= \left(
\begin{array}{ccc}
&&1\\
&\adots&\\
1&&
\end{array}\right)\,.
$$

Alors
\begin{equation} \label{X1}
\widehat{\mathfrak{h}}= \{\mathrm{diag}(X_1,\ldots, X_{n+1},0,-X_{n+1},\ldots,-X_1) \}
\end{equation}
en est une sous-algèbre de Cartan, et un caractère infinitésimal de $G$ s'identifie, par l'isomorphisme d'Harish-Chandra, à un élément $P$ de $\widehat{\mathfrak{h}}/\Omega$, $\Omega$ opérant comme auparavant sur $(X_1,\ldots , X_{n+1})$. On renvoie à \cite[\S~6.3, 6.4]{Asterisque} pour ces constructions (décrites là pour $\widehat{G}$ symplectique ou orthogonal pair).

Soit $\Omega^k(Y)$ l'espace des $k$-formes $C^\infty$, et $\Delta_k$ le laplacien. Sur une forme $\alpha\in\Omega^k$ provenant d'une représentation irréductible $\pi$ \cite[Ch.~1]{Asterisque}, $-\Delta_k$ opère par l'opérateur de Casimir $C\in \mathcal{Z}(G) \cong S(\mathfrak{h})^\Omega$. On a
$$
\pi(C) = h(C)(P)
$$
où $h$ est l'isomorphisme d'Harish-Chandra, et $P\in \widehat{\mathfrak{h}}/\Omega$. (Rappelons que nous considérons le laplacien \emph{positif}.) 
Un calcul simple, que nous ne reproduirons pas, donne alors le résultat suivant. Notons $P\in \widehat{\mathfrak{h}}$ le paramètre complet, de dimension $(2n+3)$, du caractère infinitésimal, donné par \eqref{X1}; soit $\langle P,P \rangle =\sum_{i=1}^{2n+1}P_i^2$. Alors

\begin{lem}\label{LX1}
On a
$$
\pi(C) = \langle P,P \rangle - \langle \rho,\rho \rangle
$$
donc $\Delta_k$ opère sur les formes correspondantes par
$$
-\langle P,P \rangle + \langle \rho,\rho \rangle.
$$
\end{lem}
Ici $\rho\in \widehat{\mathfrak{h}}=\mathfrak{h}^*$ est la demi-somme des racines, égale à $(n+1,n,\ldots 1,0,\break -1,\ldots -(n+1))$ de sorte que $\langle \rho,\rho\rangle =2\sum_{x=1}^{n+1}x^2$.

Rappelons que le dual unitaire de $G$ se compose des séries discrètes, des limites de séries discrètes, de représentations irréductibles de la série principale unitaire et de quotients de Langlands de la série principale non unitaire. Pour une description explicite voir Baldoni-Silva \cite[Théorème 7.1]{BaldoniSilva}. Les seules représentations de la série discrète qui interviennent ici sont celles dont le caractère infinitésimal est égal à celui, $\rho$, de la représentation triviale (voir le thèse de Pedon \cite{PedonThese} ou \cite[p.~39]{Asterisque}). Elles sont au nombre de 
$|\Omega/\Omega_K|=n+1$, et contribuent la valeur propre $0$ au spectre, 
uniquement en le degré médian $=2n$.

On notera que ces représentations ne sont pas isolées dans le spectre unitaire, comme il résulte facilement de l'article de Vogan \cite{VoganParkCity}; voir aussi \cite{JIMJ}.

Les limites de séries discrètes, et les représentations de la série discrète unitaire, contribuent (quand elles contiennent les $K$-types associés aux formes différentielles) au spectre tempéré. Les valeurs propres correspondantes sont décrites par Pedon \cite{Pedon}; d'après des résultats connus, chacune peut être approchée (en degré $k$) par une valeur propre apparaissant dans $Y$ pour un $\Gamma$ convenable. Elles feront donc partie du spectre automorphe. Cf. \cite[Théorème~2.4.6]{Asterisque}.

Décrivons les séries principales. Le groupe $G$ possède un unique parabolique non trivial, le parabolique minimal $P=\GM N$. Le groupe (algébrique) $\GM$ s'identifie naturellement (vu la définition de $G$ comme groupe unitaire quaternionien)~à
$$
B^\times \times \mathrm{Sp}(n-1)
$$
où $B$ est l'algèbre de Hamilton. Son sous-groupe compact maximal est donc $M=\SU (2) \times \mathrm{Sp}(n-1)\cong \mathrm{Sp}(1) \times \mathrm{Sp}(n-1)$. Si on note $\det : B^\times \to \RR_+^\times$ la norme réduite, une représentation de la série principale s'écrit~donc
\begin{equation}\label{X2}
\mathrm{ind}_P^G(\tau_1 |\det|^s \otimes \tau_2) := I(\tau,s)
\end{equation}
où $s \in \CC$. On écrira par la suite $M=\mathrm{Sp}(1)\times \mathrm{Sp}(n-1)=M_1\times M_2$ ; $\tau_i$ est une représentation irreducible de $M_i$.

Le sous-groupe parabolique $\widehat{P}$ de $\widehat{G}$ associé à $P$ a une composante de Levi $\widehat{M}$ qui se décrit de la façon suivante.
\begin{equation}
\widehat{M}=\left(
\begin{array}{ccc}
g_1 &&\\
&g_2 &\\
&&w {}^t g_1^{-1}w
\end{array}\right) \cong \GL(2,\CC) \times \SO(2n-1,\CC)
\end{equation}
où $g_1\in \GL(2,\CC)$, $g_2\in \SO(2n-1,\CC)$ (pour la forme de matrice antidiagonale) et $w=\left(\begin{smallmatrix}
&1\\ 1&\\
\end{smallmatrix}\right)$.

Décrivons les paramètres de Langlands des représentations \eqref{X2}. Ils sont donnés~par
$$
\varphi : W_\RR \longrightarrow \widehat{M}
$$
qui se décompose en $(\varphi_1,\varphi_2)$. Le paramètre $\varphi_2$ est associé à une représentation du groupe \emph{compact} $\mathrm{Sp}(n-1)$. D'après Langlands \cite{Langlands}, il se décrit de la façon suivante. On a pour $z\in \CC^\times \cong W_\CC \subset W_\RR$
$$
\varphi_2(z) = ((z/\overline{z})^{p_2}, 
(z/\overline{z})^{p_3},\ldots,(z/\overline{z})^{p_n},1,(z/\overline{z})^{-p_n},\ldots,(z/\overline{z})^{-p_2})
$$
où $p_2>p_3>\ldots>p_n>0$ sont des entiers. Le paramètre $P_2'=(p_2,\ldots,p_{n-1})$ définit le plus haut poids $m^2$ d'une représentation $\tau_2$ de $M_2$ de la façon suivante. Le réseau des poids de $\mathrm{Sp}(n-1)$ s'identifie naturellement à $\ZZ^{n-1}$. (Pour une description explicite, voir Pedon \cite[p.~233]{Pedon}.) La demi-somme des racines $\rho_2$ est $(n-1,\ldots 1)\in \ZZ^{n-1}$. Le plus haut poids de la représentation~$\tau_2$~est
\begin{equation*}
\begin{split}
m^2 &=(p_2 -n+1,\ p_3-n+2,\ldots p_n -1)\\
&=(m_2,\ldots m_n)
\end{split}
\end{equation*}
de sorte que $m_2\geq m_3\geq\ldots\geq 0$.

Il reste à décrire $\varphi_2(j)$ où $j\in W_\RR$ et $j^2=-1\in W_\CC$. La représentation étendue à $W_\RR$ préserve nécessairement les sous-espaces associés à $(z/\overline{z})^{p_i}$, $(z/\overline{z})^{-p_i}$ ; elle y est isomorphe à $\mathrm{ind} _{W_\CC}^{W_\RR}((z/\overline{z})^{p_i})$. D'après la formule pour le déterminant d'une induite, celui-ci est égal au transfert $\mathrm{transf}_{W_\CC}^{W_\RR}((z/\overline{z})^{p_i})$ multiplié par le caractère d'ordre 2 $\varepsilon_{\CC/\RR}$ de $\RR^\times$. Le transfert s'identifie au caractère $(z/\overline{z})^{p_i}|_{\RR^\times}=1$ de $\RR^\times$. Le déterminant est donc égal à $\varepsilon_{\CC/\RR}=\mathrm{sgn}$. Puisque l'image de $\varphi_2$ est contenue dans $\SO(2n+1)$, ce paramètre est isomorphe comme représentation de $W_\RR$~à
$$
\mathrm{ind}((z/\overline{z})^{p_2}) \oplus \ldots \oplus \mathrm{ind}((z/\overline{z})^{p_n}) \oplus \mathrm{sgn}^{n-1}.
$$

Décrivons $\varphi_1$. Il correspond à une représentation de $B^\times$, donc à une représentation de la série discrète de $\GL(2,\RR)$, non unitaire en général. On a donc
$$
\varphi_1 = \mathrm{ind}_{W_\CC}^{W_\RR}(z^{p_1}(\overline{z})^{q_1})
$$
où $p_1$, $q_1\in \CC$, $p_1-q_1\in \ZZ -\{0\}$. En particulier $|p_1-q_1|>0$ ; $|p_1-q_1|-1$ définit le plus haut poids $m_1$ d'une représentation de $M_1=\mathrm{Sp}(1)$ : c'est la représentation $\tau_1$. Le paramètre $s$ de \eqref{X2}~est
\begin{equation} \label{X3}
s=\frac{p_1+q_1}{2}.
\end{equation}

On écrira souvent $z^p\overline{z}^q$ pour le caractère de $\CC^\times$, qui définit donc la représentation induisante de~$B^\times$.

Le paramètre infinitésimal associé à $I(\tau,s)$ est, d'après un calcul standard \cite[Lemme~6.3.1]{Asterisque} :
$$
P=(p_1,q_1,P_2,-q_1,-p_1)
$$
$$
P_2=(p_2,\ldots p_{n},0,-p_{n},\ldots -p_2),$$
donc, d'après le lemme \ref{LX1}, la valeur propre
du laplacien associée, quand la représentation $\pi=I(\tau,s)$ intervient dans les $k$-formes, est
\begin{equation} \label{X5}
\lambda_\pi = -\langle P,P \rangle + \langle \rho,\rho \rangle = -2 (p_1^2 + q_1^2) -\langle P_2,P_2 \rangle + \langle \rho,\rho \rangle.
\end{equation}

Vérifions la compatibilité de notre expression avec celle de Pedon \cite[p.~237]{Pedon}. On vérifie aisément que le paramètre induisant \og$\lambda$\fg $\in\CC$ de Pedon pour la série principale \cite[\S~3.2]{Pedon} est égal à $2is$ dans notre notation. Pour le laplacien positif, il obtient alors la valeur propre
$$
\lambda_\pi = -4 s^2 + \rho_N^2 - c(\tau)
$$
où $\rho_N$ (dans le dual de l'algèbre de Lie du centre connexe $A$ de $B^\times$) est, dans sa paramétrisation, égal à $2n+1$, et $c(\tau)$ est la valeur propre de l'opérateur de  Casimir dans $\tau$. On a $c(\tau)=c(\tau_1)+c(\tau_2)$,
$$
\begin{array}{ll}
c(\tau_2) &=\langle P_2,P_2 \rangle - \langle \rho_2,\rho_2 \rangle \\
c(\tau_1) &=m_1^2 -1
\end{array}
$$
d'où l'expression de Pedon :
\begin{equation} \label{X6}
\lambda_\pi = -4s^2 +(2n+1)^2 -\langle P_2,P_2 \rangle - m_1^2 +\langle \rho_2,\rho_2 \rangle +1.
\end{equation}

Or $p_1=\frac{m_1}{2}+ s$, $q_1=-\frac{m_1}{2}+s$, $p_1^2+q_1^2=2s^2+\frac{m_1^2}{2}$, et l'égalité de \eqref{X5} et \eqref{X6} résulte alors de
$$
\langle \rho,\rho \rangle = \langle \rho_2, \rho_2\rangle +(2n+1)^2+1
$$
{\it i.e.} de
$$
2(n^2+(n+1)^2) = (2n+1)^2+1.
$$

\subsection{} En imitant le chapitre 6 de \cite{Asterisque}, nous décrivons maintenant, à la place réelle, les paramètres d'Arthur des représentations apparaissant dans  $L^2(\Gamma\backslash G)$. L'argument de réduction déjà utilisé dans \cite{Inventiones} et qui sera résumé dans le \S~\ref{SX5} nous ramène à considérer la forme quasi-déployée $G^*$ de $G$ : c'est le groupe orthogonal déployé d'un espace de dimension $2n+3$. On considère un paramètre d'Arthur $\psi$ pour $ G^*$ (sur $\RR$). Rappelons qu'un tel paramètre est de la forme
$$
\psi : W_\RR \times \SL(2,\CC) \longrightarrow \widehat{G}.
$$
En particulier on a aussi
$$
\psi : W_\RR \times \SL(2,\CC) \longrightarrow \GL(N,\CC).
$$

On écrira simplement $S$ pour $\SL(2,\CC)$. Un tel paramètre s'écrit sous la forme
$$
\psi = \bigoplus_{i=1}^r \rho_i \otimes r_i
$$
où $\rho_i$ est une représentation irréductible de $W_\RR$ et $r_i$ une représentation irréductible de~$S$.

Considérons pour l'instant la restriction de $\psi$ à $W_\CC \times S$, $W_\CC = \CC^\times$. On peut écrire, changeant de notation :
\begin{equation}\label{X7}
\psi_\CC = \bigoplus_{i=1}^r \chi_i \otimes r_i
\end{equation}
où $\chi_i= z^{p_i} \overline{z}^{q_i}$ est un caractère de $\CC^\times$. Un tel paramètre sera obtenu par restriction à la
place réelle d'un paramètre relatif à $F$ (Arthur \cite{Arthur}). En particulier, 
dans \eqref{X7}, les caractères $\chi_i$ correspondent aux caractères induisants du facteur en $v_0$ d'une représentation cuspidale de $\GL(a,\Ab_F)$, dont on a pris le changement de base à $\CC/\RR$. Ils vérifient donc la majoration de Luo-Rudnick-Sarnak étendue dans \cite[Chapitre~7]{Asterisque} :
$$
|\chi_i(z)|^{\pm1} \leq (z\overline{z})^{\eta_N}.
$$

Rappelons que $\eta_N =\frac{1}{2} -\frac{1}{N^2+1}$ avec ici $N=2n+3$, le fait essentiel étant que $\eta_N<\frac{1}{2}$. Dans ce qui suit, les résultats conditionnels (supposant la conjecture de Ramanujan généralisée) correspondent donc à~$\eta_N=0$.

Si $\chi(z) =z^{p}(\overline{z})^q$, apparaît dans \eqref{X7}, on sait donc que
\begin{equation} \label{X8}
|\mathrm{Re}\big(\frac{p+q}{2}\big)| \leq \frac{1}{2}-\frac{1}{N^2+1} = \eta_N. 
\end{equation}

Il sera commode de noter $R$ le domaine de $\CC$ défini par cette inégalité :
\begin{equation} \label{X9}
R=\{s\in \CC \; : \;  |\mathrm{Re} (s)| \leq \eta_N \}.
\end{equation}

Revenons à \eqref{X7} et supposons que $\psi_\CC$ est restreint d'un paramètre $\psi$ pour $\widehat{G}$. Le paramètre $\psi$, ainsi que les $r_i$, étant autoduaux, $\chi_i^{-1}\otimes r_i$ apparaît nécessairement avec $\chi_i\otimes r_i$. Si $\chi_i =\chi_i^{-1}$, $\chi_i$ est trivial. Donc $\psi_\CC$ est la somme des termes
\begin{equation}\label{X10}
\sum_{\chi_i\not= 1}
 (\chi_i \otimes r_i) \oplus (\chi_i^{-1}\otimes r_i)
 \end{equation}
 et
\begin{equation} \label{X11}
\sum_i (1\otimes r_i).
\end{equation}
Par ailleurs, le paramètre est orthogonal, pour une dualité non dégénérée. Si $r_i$ est symplectique (donc de degré $n_i$ pair) et apparaît dans \eqref{X11}, sa multiplicité est donc paire.
 
Puisque le paramètre s'étend à $W_\RR$, un caractère $\chi_i(z)$ doit apparaître avec $\chi_i(\overline{z})=\chi_i^c$. Si $\chi_i=z^{p_i}(\overline{z})^{q_i}$ apparaît dans 
\eqref{X10}, et si $p_i\not= q_i$, il doit aussi apparaître 
$\chi_i^c= z^{q_i}\overline{z}^{q_i}$.
Ce caractère est égal à $\chi_i^{-1}$ si $\chi_i^c=\chi_i^{-1}$ (on dira alors 
que $\chi_i$ est de type unitaire), c'est-à-dire si $p_i=-q_i$ ; noter qu'alors $p_i,q_i\in\frac{1}{2}\ZZ$. On voit alors que \eqref{X10} se subdivise en les cas suivants (on dira que $\chi$ est réel si $\chi=\chi^c$) :
\begin{align}
 &\sum_{\chi_i\ \textrm{r\'eel}} (\chi_i\otimes r_i) \oplus (\chi_i^{-1}\otimes r_i)\\
& \sum_{\chi_i\atop \mathrm{de\ type\ unitaire}} (\chi_i \otimes r_i\oplus \chi_i^{-1}\otimes r_i)\\
& \sum_{\mathrm{autres\ }\chi_i}(\chi_i\otimes r_i \oplus\chi_i^{-1} \otimes r_i \oplus \chi_i^c \otimes r_i \oplus \chi_i^{-c}\otimes r_i) \label{X14}
\end{align}
(Noter qu'un caractère réel de type unitaire est trivial !)

\`A toute composante $\chi\otimes r$ sont associées $n$ composantes du caractère infinitésimal $P$ de la représentation, données si $\chi=z^p \overline{z}^q$ par
$$
\left(p+\frac{n-1}{2},p+\frac{n-3}{2}, \ldots, p+ \frac{1-n}{2}\right)
$$
où $n$ est le degré de $r$. Notons pour simplifier $p[n]$ cette famille de $n$ r\'{e}els. Rappelons aussi que $P=(p_1,p_2,\ldots, -p_1)$ peut être paramétré par ses $(n+1)$ premières composantes (modulo $\Omega$), la composante médiane étant nulle. On note $P'$ cet élément de $\CC^{n+1}$. On voit alors que la contribution de \eqref{X10} à $P'$ est la somme
\begin{align} \label{X15}
 &\sum_{\chi_i\ \textrm{r\'eel}\atop \mathrm{ou\ de\ type\ unitaire}} p_i[n_i] &(a)\\
& \oplus \sum_{\mathrm{autres\ }\chi_i}(p_i[n_i],q_i[n_i]) &(b)\nonumber
\end{align}
Notons simplement $[n]$ la séquence $0[n]=\left(\frac{n-1}{2},\ldots,\frac{1-n}{2}\right)$. Considérons les termes \eqref{X11}. Si $n_i$ est pair, $1\otimes r_i$ contribue deux fois $[n_i]$ à $P$, et une fois à $P'$. Si $n_i$ est impair, en notant $[n]^+$ la séquence $\left(\frac{n-1}{2},\frac{n-3}{2},\ldots 0\right)$, on voit que $1\otimes r_i$ contribue $[n_i]^+$ à $P'$. Noter que le nombre de termes \eqref{X11} avec $n_i$ impair est nécessairement impair (donc non nul). Il y a donc un terme redondant en $0$ (correspondant au terme médian de $P$) que l'on retranche par la notation formelle $\ominus 0$. Au bout du compte $P'$ est la somme de \eqref{X15} et~de
\begin{align} \label{X16}
 &\sum_{n_j\ \mathrm{pair}} \ [n_j]  &(c)\\
& \oplus \sum_{n_j\ \mathrm{impair}}[n_j]^+ &(d)\nonumber\\
&\ominus 0\,.\nonumber
\end{align}

Nous devons maintenant déterminer quand ces représentations s'étendent à $W_\RR$. Rappelons que les sous--groupes de Levi de $\widehat{G}$ sont de la forme
$$
\widehat{M} = \GL(a_1,\CC) \times \GL(a_2,\CC) \times\ldots \times \SO(2m+1,\CC)
$$
où $m\ge 0$, $2n+3=2m+1+2\sum a_i$ et $\widehat{M}$ est plongé dans $\widehat{G}$ de façon analogue à~\eqref{X3}.

En particulier, tous les termes de \eqref{X10} et \eqref{X11} qui peuvent s'écrire (comme représentations de $W_\RR \times S$, après extension à $W_\RR$) sous la forme 
$\rho_1 \oplus \tilde{\rho}_1$ s'étendent en des param\`{e}tres $W_\RR \times S \to \widehat{G}$ en passant par un facteur $\GL(a_i)$, et on est ramené alors, pour le paramètre
restant, à un morphisme $W_\RR \times S \to \SO(2m+1,\CC)$, $m<n+1$. Dans \eqref{X14}, $(\chi_i \oplus \chi_i^c)$ s'étend à $W_\RR$ (pas d'obstruction à la descente pour $\GL(2))$ et \eqref{X14} s'écrit alors ainsi. De même évidemment pour les termes $(c)$ de \eqref{X16}, ainsi que pour les représentations contribuant à $(d)$ avec multiplicité paire.
Il reste donc les termes $(a)$ de \eqref{X15} ainsi que les termes $(d)$ de \eqref{X16}, avec multiplicité~1.

Puisque $n_j$ est impair, la représentation $r_j$ de $S$ est orthogonale. On peut même la multiplier par le caractère $\mathrm{sgn}$ de $W_\RR$, donnant alors une 
représentation $W_\RR \times S \stackrel{\rho}{\longrightarrow} \mathrm{O}(n_j)$ telle que $\det(\rho(j))=-1$. A conjugaison près, on peut alors plonger $\prod \SO(n_j)$ dans $\widehat{G}$, correspondant à une décomposition orthogonale de $\CC^N$. On a $(\prod \mathrm{O}(n_j) \times \mathrm{O}(m))^+ \subset \widehat{G}$, le $+$ 
correspondant au déterminant total égal~à~1.

Noter que $m$ est pair puisque le nombre de termes $(d)$ est impair. En tenant compte des cas déjà  décrits, on est donc ramené à étendre le paramètre à valeurs dans 
$\mathrm{O}(m)$, somme des termes $(a)$ de \eqref{X15}, à~$W_\RR$.

Si $\chi_i$ est réel, $\chi_i(z)=(z\overline{z})^p$ s'étend à $W_\RR$, intervient avec son dual, et on est de nouveau ramené au cas précédent. Supposons enfin $\chi$ de 
\emph{type unitaire} :
$$
\chi(z) =(z/ \overline{z})^p,\ p\in \frac{1}{2} \ZZ.
$$
On cherche à étendre la somme des $(\chi_i\oplus \chi_i^{-1})\otimes r_i$ en une représentation vers $\mathrm{O}(m)$. Des arguments évidents d'isotypie montrent 
que l'on peut se réduire à un unique $\chi$, puis à un unique $r$, puis d'après l'argument précédent à $(\chi \oplus \chi^{-1}) \otimes r$ apparaissant avec multiplicité~1.

On a $\chi^{-1}=\chi^c$ et $\chi\oplus \chi^{-1}$ s'étend de façon unique à $W_\RR$ en $\mathrm{ind}_{ W_\CC}
^{W_\RR}(\chi)$, celle-ci étant irréductible. Elle est symplectique si $p\in \frac{1}{2}+\ZZ$, orthogonale si $p\in \ZZ$. Mais la forme symétrique sur l'espace de 
$\mathrm{ind}(\chi)\otimes r$ est nécessairement, par irréductibilité des facteurs, le produit tensoriel d'une forme bilinéaire sur $\mathrm{ind}(\chi)$ et d'une sur $r$, chacune invariante. On en déduit~enfin :

\begin{lem} \label{LX2}
Les paramètres d'Arthur $W_\CC \times S \to \widehat{G}$ sont les sommes de paramètres de la forme \eqref{X11}-\eqref{X14} vérifiant la condition suivante : 
pour un paramètre $\chi\otimes r\oplus \chi^{-1} \otimes r$ où $\chi=(z/\overline{z})^p$ est de type unitaire, et de multiplicité impaire, 
$$
\begin{array}{llll}
\dim(r) &\mathit{est\ impaire} &\mathit{si} &p\in\ZZ\\
\dim(r) &\mathit{est\ paire} &\mathit{si} &p\in\frac{1}{2}+\ZZ.
\end{array}
$$
\end{lem}

Noter qu'il y a en fait des restrictions plus précises sur les paramètres réels : dans le calcul précédent, selon les choix des extensions des facteurs $1\otimes r_j$ à $W_\RR$, la représentation finale serait à image dans $\mathrm{O}^-(m,\CC)$. Ceci n'a pas d'importance pour nous. 

\subsection{} Nous pouvons maintenant déterminer, dans les différents cas possibles, les bornes sur les valeurs propres de $\Delta_k$ résultant des calculs précédents. Rappelons qu'on~a
$$
P' =(p,q,P_2')
$$
où on peut supposer que (modulo $\tilde{W})$
$$
P_2'=(p_2>p_3>\cdots>p_n>0).
$$
Il y a donc $(n-1)$ coordonnées entières distinctes et non nulles dans les termes \eqref{X15} et \eqref{X16}, on en déduit que

\begin{lem} \label{LX3}
Il apparaît au plus un terme de type $(c)$; dans ce cas on a $n_j=2$ et le terme de type $(c)$ est de multiplicité~$1$.
\end{lem}
Nous considérons maintenant les différents cas possibles.

\subsubsection{Cas A}
On suppose que le terme désigné par le lemme précédent intervient. Les autres coordonnées de $P'$ sont entières et distinctes. Dans \eqref{X16}, il intervient donc au plus un terme $(d)$ avec $n_j>1$ (sinon il y a deux coordonnées égales à $1$), et au plus un terme avec $n_j=1$ : c'est alors l'unique terme $(d)$. Ainsi \eqref{X16}$(d)$ se réduit~à
\begin{equation} \label{X17}
[m]^+ \quad (m\ge 1)
\end{equation}
ce terme étant vide si $n_j=1$. D'après \eqref{X15}, on a~donc
$$
P' =\left(\frac{1}{2}, -\frac{1}{2},\frac{m-1}{2},\frac{m-3}{2},\ldots,1,\ p_i[n_i],\ p_j[n_j],q_j [n_j]\right)
$$
où les $p_i \ (\not=0)$ sont associés aux caractères réels ou de type unitaire. Si $p_i$ provient d'un caractère réel, $p_i\in R-\{0\}$ (cf. \eqref{X9}) 
et alors $p_i[n_i]$ ne contient pas d'entiers : c'est impossible. Si $p_i$ provient d'un caractère de type unitaire $\chi_i$, celui-ci 
intervient avec multiplicité 1 et le lemme~\ref{LX2} s'applique.

Les éléments de $p_i[n_i]$ sont donc des entiers. Considérons un terme $(p_j[n_j],\break q_i[n_j])$. On a $q_j=m_j+p_j$ $(m_i\in \ZZ)$ et $(p_j[n_j], q_j[n_j]+m_j)$ 
est formé d'entiers distincts. C'est possible ; chacun de ces termes ({\it i.e.}, chaque $\chi_j$) apparaît du reste avec multiplicité 1. Dans ce cas on doit avoir
$$
p_j,q_j \equiv \frac{n_j-1}{2} \ [1]
$$
par intégralité. Mais $p_i+q_j\in 2R$, et ces deux conditions impliquent $p_i+q_j\in ]-1,1[$ et $p_i+q_j$ entier. Donc $p_i+q_j=0$ et le caractère~est
$$
(z/\overline{z})^{p_j},\quad p_j \equiv \frac{n_i-1}{2} \ [1]
$$
de type unitaire. On peut donc absorber ces cas dans le cas unitaire.

Calculons la valeur propre correspondante. On~a
\begin{align*}
\lambda &=-\langle P,P \rangle + \langle \rho,\rho \rangle \\
&=-1- \langle P_2,P_2 \rangle  +\langle \rho,\rho \rangle ,
\end{align*}
le poids $m_1$ de $\tau_1$ (\S~\ref{X1}) étant d'ailleurs égal à $1$. Si $\pi$ intervient dans les $k$-formes, $(m_1,P_2')$ est donc le poids 
d'une représentation de $M$ associée aux $k$-formes ; le paramètre $s=1/2-1/2$ est nul. On est dans le cas tempéré et donc

\begin{lem}
Dans le cas $A$, on a
$$
\lambda \ge \alpha_k
$$
si $\pi$ contribue à $\Omega^k$, $\alpha_k>0$ étant la borne tempérée. De plus il existe un $M$-type associé à $\Omega^k$ (Pedon \cite{Pedon}) et $M_1$-trivial.
\end{lem}

On suppose maintenant qu'il \textit{n'existe pas} de terme de type~$(c)$.

Les termes réels de $(a)$ $(p_i\in R)$ donnent des éléments non entiers, donc il existe au plus un terme, avec $n_i=2$, contribuant $(p_i+\frac{1}{2}$, $p_i-\frac{1}{2})$ nécessairement égal à $(p,q)$, ou bien deux termes $(p_1,p_2)=(p,q)$. Considérons successivement les deux possibilités.

\subsubsection{Cas B}
Un terme réel de $(a)$ avec $n_i=2$. 
Dans ce cas, les arguments précédents s'appliquent aux termes restants, de type $(b)$ ou $(d)$. On a donc
\begin{align*}
\lambda &=-2(p^2+q^2) - \langle P_2,P_2 \rangle +\langle \rho,\rho \rangle \\
&=-4 p_i^2 -1 - \langle P_2,P_2 \rangle + \langle \rho,\rho \rangle.
\end{align*} 
On sait que $p_i\in R$ ; $\lambda$ étant réel on a en fait $p_i\in i\RR \cup] -\eta_N,\eta_N[$, correspondant à un segment de série complémentaire pour $I(\tau,s)$. Noter que $m_1$ est toujours égal à 1. Ainsi

\begin{lem} \label{LX5}
Dans le cas $B$, on a dans $\Omega^k$ :
$$
\lambda \ge \alpha_k -4 \eta_N^2 = \alpha_k -\left(\frac{N^2-1}{N^2+1}	\right) > \alpha_k -1 >0.
$$
Il existe un $M$-type associé à $\Omega^k$ et $M_1$-trivial.
\end{lem}

Supposons maintenant qu'il existe des termes réels de $(a)$ avec~$n_i=1$.

Dans ce cas, $p_1$, $p_2\in R$ ne sont pas entiers, donc $\{p,q\} = \{p_1,p_2\}$. Alors $p-q\in \ZZ$ donc $p-q=0$ si $p,q\in R$. Mais ceci est impossible (le caractère induisant $\chi$ est ramifié, cf.~\S~\ref{X1}).

On est donc ramené au cas où il n'y a ni terme $(c)$ ni terme $(a)$ réel.

Considérons les termes $(d)$ de \eqref{X6}. Puisque $P_2'$ a $(n-1)$ coordonnées non nulles et distinctes, on voit aussitôt qu'il existe au plus un terme tel que $n_j\ge 3$, complété peut-être par deux termes $n_i=1$, ou bien un ou trois termes~$n_i=1$.

Considérons le terme $(b)$ de \eqref{X15}. On a par hypothèse $\mathrm{Re}(p_i+q_i)\in ]-1,1[$. Si le terme associé de $(b)$ contient un entier, on a $p_i+x_i\in \ZZ$ donc $q_i+x_i$ et $q_i-x_i\in \ZZ$ où $x_i$ apparaît dans $[n_i]$. Alors $p_i+q_i\in \ZZ$, contrairement au fait que $p_i+q_i\not=0$ par hypothèse. Donc les termes $(b)$ ne peuvent contribuer à $P_2'$. Il y a donc au plus un terme $(b)$, avec $n_i=1$, et alors $(p_i,q_i)=(p,q)$.

\subsubsection{Cas C}
Un terme $(b)$, $(p,q)=(p_i,q_i)$. Tous les autres termes de $P'$ sont alors entiers et  de valeurs absolues distinctes. Il y a donc au plus un unique terme $(d)$, contribuant $(k,\ldots ,1)$ avec $k=\frac{n_i-1}{2}\ge 1$ (ou le terme $\ominus0$). Le paramètre $P_2$ est obtenu par concaténation des $p_i[n_i]$ de type unitaire et de cet intervalle (ou $\emptyset$). Puisque $s=\frac{p_i+q_i}{2}\in R-\{0\}$, on~a
\begin{align*}
\lambda &=-2(p^2+q^2) -\langle P_2,P_2 \rangle  + \langle \rho,\rho \rangle \\
&= -m_1^2 -4s^2 - \langle P_2,P_2 \rangle  + \langle \rho,\rho  \rangle
\end{align*}  
de sorte qu'on a la minoration analogue au lemme \ref{LX5}, que l'on ne reproduit pas. Dans ce cas $m_1$ peut être~$>1$.

\subsubsection{Cas D}
Pas de terme $(b)$. Dans ce cas on n'a que des termes $(d)$ et des termes $(a)$ de type unitaire. Noter que d'après le lemme~\ref{LX2}, ceux-ci contribuent des entiers sauf peut-être si $\chi$ apparaît avec une multiplicité paire.

Comme expliqué dans l'introduction c'est le cas où nos résultats sont les plus faibles. Nous considérons les termes $(a)$ de type unitaire :
$$
\sum (z/\overline{z})^{p_i} \otimes r_i = \sum \chi_i \otimes r_i \quad (p_i\not= 0)
$$
donnant un paramètre
$$
\sum p_i[n_i].
$$

Sachant que deux termes au plus de $P'$ ont la même valeur absolue, on voit~que
\begin{equation}
\mathrm{mult}(\chi_i\otimes r_i) \le 1 \quad \mathrm{sauf\ si}\quad \dim(r_i)\le 2.
\end{equation}

Si $\dim(r_i)=2$, on obtient deux coïncidences donc de~plus
\begin{equation}
\begin{split}
& \mbox{Il existe au plus un terme avec } \dim(r_i)=2\\
& \mbox{et alors } \mathrm{mult}(\chi_i\otimes r_i)=1.
\end{split}
\end{equation}
Alors le terme $\chi \otimes r\oplus \chi^{-1}\otimes r$ est de multiplicité impaire au sens du lemme~\ref{LX2}, et l'on a donc :
$$
p_i + 1/2 \in \ZZ ,
$$
ceci étant donc vrai dans tous les cas. Il en résulte que tous les termes de type $(a)$ sont entiers. Puisque c'est aussi vrai des coordonnées des termes $(d)$, on a enfin : 

\begin{lem}
Dans le cas $D$, on a pour toute valeur propre non nulle dans~$\Omega^k$ :
$$
\lambda = -2(p^2+q^2) - \langle P_2,P_2 \rangle + \langle \rho,\rho \rangle
$$
qui est un entier pair $\ge2$.
\end{lem}

On a donc enfin :

\begin{thm}
Soit $\lambda>0$ une valeur propre du laplacien dans $\Omega^k(U)$. Alors
\begin{enumerate}
\item $\lambda$ est un entier pair $>0$, ou
\item $\lambda\ge \alpha_k -\left(\frac{N^2-1}{N^2+1}\right)^2$
\end{enumerate}
où $\alpha_k$ est la borne de Pedon \cite{Pedon}.
\end{thm}
Dans le degré médian $(k=2n)$ et ses voisins $(k=2n-1,2n+1)$, on a $\alpha_k=1$. La partie (i) du théorème est donc vide. Pour $k\le 2n-2$, $\alpha_k\ge 4$.

\subsection{Le cas des fonctions (cas sphérique)}
Dans le cas des fonctions $\Omega^0=C^\infty(Y)$, on obtient le résultat suivant. On~a
$$
P_2 = \rho_2 =(n-1,\ldots ,1,0,\ldots , 1-n)
$$
donc (\S~\ref{X1}) 
$$-\langle P_2,P_2 \rangle + \langle \rho,\rho \rangle =(2n+1)^2+1.$$

Par ailleurs $m_1=1$. La borne tempérée est donc :
$$
\alpha_0 = (2n+1)^2
$$
et $\lambda =-2(p^2+q^2)+(2n+1)^2+1$.

Dans ce cas, le calcul de $D$ donne l'estimée suivante.

Ecrivons, à l'ordre près, $(p,q)=\left(\frac{1}{2}+s, -\frac{1}{2}+s\right)$ où $s\in \frac{1}{2}+\ZZ$. Alors
\begin{align*}
\lambda &=(2n+1)^2-4s^2\\
&=(2n+1)^2-k^2
\end{align*}
où $k=1,3,\ldots , 2n+1$ (valeur maximale, correspondant à la représentation triviale). On a donc :

\begin{thm} \label{T38}
Si $\lambda$ est une valeur propre du laplacien dans $C^\infty(Y)$,
\begin{align*}
&\lambda = (2n+1)^2-k^2\quad (k=1,3,\ldots 2n+1)\\
\mathrm{ou} \hskip2cm &\lambda\ge (2n+1)^2 -\Big(\frac{N^2-1}{N^2+1}\Big)^2 \quad (N=2n+3).
\end{align*}
\end{thm}
Si on suppose la conjecture de Ramanujan, on peut  remplacer la deuxième alternative par
$\lambda \geq (2n+1)^2$. Le théorème est alors optimal. Il découle en effet de \cite{Faraut} que chaque représentation sphérique de caractère infinitésimal
$$(a , a-1 , \rho_2 , 1-a , a), \quad a = 1 , 2 , \ldots , n+1,$$
intervient discrètement dans la représentation régulière $L^2 (\mathrm{Sp} (n-1 , 1) \backslash \mathrm{Sp} (n,1))$.\footnote{Faraut \cite{Faraut} note $\pi_{\rho + 2r}$ cette représentation, avec $\rho = 2n+1$ et $0 < \rho + 2r < \rho$; de sorte que $a=n+r+1$ et $\lambda_\pi = \rho^2 - (\rho+2r)^2= (2n+1)^2 - (2a-1)^2$.} \ 
Comme dans \cite[Example B]{BLS}, on en déduit que ces représentations appartiennent au dual automorphe de $\mathrm{Sp} (n,1)$. Puisqu'elles sont par ailleurs isolées d'après le théorème \ref{T38}, on conclut que ces représentations interviennent dans un $L^2 (\Gamma \backslash \mathrm{Sp} (n,1))$ pour un certain groupe de congruence $\Gamma$.

\subsection{Complétion de l'argument global} \label{SX5}
Nous revenons à l'argument global de réduction au cas déployé annoncé au début de \ref{X2}. C'est de nouveau celui de \cite{Inventiones}, à la différence près qu'ici d'autres types de groupes classiques apparaissent comme groupes endoscopiques. 

Revenons aux notations du début de \ref{X1}, de sorte que $\G$ est un groupe semi-simple sur $F$. On note $\G^*=\SO^*(2n+3,F)$ sa forme intérieure déployée. Le lecteur peut supposer $\G$ anisotrope : l'argument supplémentaire si $\G$ est de rang 1 sur $F$ (et $F$ est donc égal à $\mathbf{Q}$) est exactement semblable à celui de la fin du \S~1  et ne sera pas répété. Pour simplifier les notations écrivons $G$, $G^*$ pour $\G$, $\G^*/F$.

Comme dans \cite[\S~4]{Inventiones}, on commence par \textit{stabiliser} la formule des traces, d'où avec les notations d'Arthur et de cet article
\begin{equation} \label{X20}
 I_{disc,t}^G(f) = \sum_H \iota(G,H) S_{disc,t}^H (f^H).
\end{equation} 
 
Les groupes $H$ sont les groupes endoscopiques elliptiques. Ils sont décrits par Arthur \cite[\S~1.2]{ArthurBook} : leurs groupes duaux sont donnés par
$$
\widehat{H} = \SO (2a,\CC)\times \SO(2b+1,\CC) \subset \widehat{G},
$$
$a+b=n+1$. Puisque $\SO(2a)$ possède des automorphismes extérieurs, une donnée supplémentaire est $\eta : \mathrm{Gal}(\overline{F}/F)\to \mathrm{Autext} \ \SO(2a) =\ZZ/2\ZZ$.\footnote{Le morphisme $\eta$ est non trivial si $a=1$.} Comme dans \cite{Inventiones}, ces données sont en nombre fini, la ramification étant fixée. La donnée de $\eta$ définit ${}^L H$ et on a naturellement~\cite{Arthur}
$$
{}^L H \longrightarrow {}^L G \longrightarrow \GL(N,\CC) \times \mathrm{Gal} (\overline{F} / F).
$$
 
Nous ne répétons pas les généralités relatives au transfert $f\leadsto f^H$ (\cite[\S~2]{Inventiones}) désormais acquis grâce aux travaux de Ngô, Waldspurger et al. En revanche, les arguments du \S~\ref{X2} relatifs aux caractères infinitésimaux nécessitent de nouveau le contrôle de celui-ci dans les paquets $\Pi (\psi)$ de représentations de $G^*(F_{v_0})$, basés sur la formule d'Arthur \cite[Théorème 2.2.1]{Arthur} :
 \begin{equation} \label{X21}
\mathrm{trace}(\Pi (\varphi)A_\theta) = \sum_\pi \varepsilon(\pi) m(\pi) \mathrm{trace}\ \pi(f)
 \end{equation}
où $\varphi$ est une fonction sur $\GL(N,\RR)$, $f$ sur $G^*(F_{v_o}) = G^*(\RR)$ lui est associée, $\pi$ décrit $\Pi (\psi)$ et les autres notations sont celles d'Arthur.
 
Comme dans la démonstration du lemme~1.1, nous pourrions utiliser que $z\varphi$ est associée à $(Nz)f$ pour 
 $$
 z\in \mathcal{Z}(\GL(N,\CC))\  \mathrm{et}\ N:\mathcal{Z}(\GL(N,\CC)) \longrightarrow \mathcal{Z} (G^*(\RR))
$$
l'application naturelle entre les centres des algèbres enveloppantes correspondant à notre identification des paramètres infinitésimaux dans \ref{X2}. Ceci se déduit des résultats généraux de Shelstad pour les groupes non connexes dans le cas tempéré \cite{Shelstad}. Plus directement, l'identité \eqref{X21}, combinée avec les formules d'intégration de Weyl, donne une identité
 \begin{equation} \label{X22}
 \Theta_{\Pi,\theta}(t) = \sum_\pi \varepsilon(\pi) m(\pi) \Theta_\pi (t')
 \end{equation}
entre caractère tordu et caractères ordinaires des représentations $\pi$. Noter que le facteur $\Delta_{IV}$ de \cite[(3.4)]{Inventiones} n'intervient pas, car $G^*$ est l'unique groupe endoscopique de dimension maximale pour $\GL(N)$ tordu. Les éléments $t$ (fortement $\theta$-régulier dans $\GL(N,\RR)$) et $t'$ (régulier dans $G(\RR)$) sont associés. On termine en utilisant l'argument de \cite[Lemme~3.4]{Inventiones}
relatif aux parties radiales, le fait que la correspondance $\mathcal{N}$ de Kottwitz-Shelstad et Waldspurger est surjective \cite[\S~3.2]{WaldspurgerDuke} et l'indépendance linéaire des~$\Theta_\pi$.\footnote{L'argument relatif à la surjectivité de la norme était omis dans \cite[Lemme 3.4]{Inventiones}. Pour ce cas voir Kottwitz-Shelstad \cite[\S 3.2-3.3]{KottwitzShelstad}. Plus précisément, ceux-ci démontrent la surjectivité si toute classe de conjugaison tordue, fortement régulière, et rationnelle, de $\GL(N)$ contient un élément réel. Le seul cas utilisé dans \cite{Inventiones} et non traité par Waldspurger \cite{WaldspurgerDuke} est celui o\`{u} $G=\SO^*(2n)$, $\widehat{G}= \SO(2n,\CC)$. Mais l'assertion résulte alors de la construction par Waldspurger de la norme pour $G=\SO^*(2n+1)$, $\widehat{G}= \Sp(n,\CC)$.}

On sait maintenant que les caractères infinitésimaux des représentations de $\Pi(\psi)$ sont définis, comme dans le \S~\ref{X2}, par $\psi$. Revenons alors à l'identité \eqref{X20}. Pour $G$ anisotrope, le membre de gauche est simplement la partie associée à $t$ de la trace de $f$~sur
$$
L^2(G(F)\backslash G(\Ab)).
$$
Comme dans \cite[\S~5]{Inventiones}, il s'agit maintenant de déstabiliser le membre de droite de \eqref{X20}, ce qui donne une formule (cf. \cite[5.5]{Inventiones})
$$
I_{\mathrm{disc},t}^G(f) = I_{\mathrm{disc},t}^{G^*}(f^*) + \sum_{\mathcal{J}}\eta(G,\mathcal{J}) I_{\mathrm{disc},t}^J(f^{\mathcal{J}})
$$
portant sur des groupes endoscopiques itérés 
$$\mathcal{J}=(G,H_1,\ldots H_r),\ H_r =J\not\cong G^*.$$

D'après la description précédente des groupe endoscopiques pour $G$ (ou $G^*$), on a alors
$$
\widehat{J} = \prod_i \SO(2 a_i,\CC) \times \prod_i \SO(2b_j+1,\CC)
$$
et ${}^L J$ est défini par la donnée de caractères d'ordre $2$, $\eta_i$, de $\mathrm{Gal}(\overline{F}/F)$. La fin de l'argument est alors identique à \cite[\S~6.2]{Inventiones}. Ceci termine la démonstration.

\bibliography{bibli}

\bibliographystyle{plain}

\end{document}